\documentclass[pstricks,10pt,british,leqno]{amsart} 

\title{Symmetric products of Galois-Maximal varieties}
\author{Javier Orts}
\email{javier.orts@tecnico.ulisboa.pt}
\address{Departamento de Matemática\\
Instituto Superior Técnico\\
Av. Rovisco Pais\\
1049-001 Lisbon\\
Portugal
}
\date{August 2024} 
\usepackage{geometry}
\usepackage{mlmodern}
\usepackage[T1]{fontenc}
\usepackage[UKenglish]{babel}

\usepackage{mathtools}
\usepackage{amssymb}
\usepackage{comment}
\usepackage{pdflscape}
\usepackage[hidelinks]{hyperref}
\usepackage{enumitem}
\usepackage{graphicx}
\usepackage{caption}
\usepackage{float}
\usepackage{subcaption}
\graphicspath{{./Figures/}}
\usepackage{pst-solides3d}
\usepackage{tikz-cd}
\usetikzlibrary{babel}

\newcommand\loccit{\emph{loc.~cit.}}
\newcommand\cf{\emph{cf.}}
\newcommand\ie{\emph{i.e.}}

\newcommand\etal{\emph{et~al.}}

\let\geq\relax
\let\geq\geqq
\let\leq\relax
\let\leq\leqq

\makeatletter
\DeclareFontEncoding{LS1}{}{}
\DeclareFontSubstitution{LS1}{stix}{m}{n}
\DeclareMathAlphabet{\mathscr}{LS1}{stixscr}{m}{n}
\makeatother

\DeclareMathOperator{\colim}{colim}

\newcommand{\pt}{pt}
\newcommand{\N}{\mathbf N}
\newcommand{\Z}{\mathbf Z}

\newcommand{\R}{\mathbf R}
\newcommand{\C}{\mathbf C}
\newcommand{\F}{\mathbf F}

\newcommand{\RO}{\mathrm{RO}}

\let\oldvarprojlim\varprojlim
\RenewDocumentCommand\varprojlim{e{^_}}{%
  {\oldvarprojlim\IfNoValueF{#2}{_{#2}}}\IfNoValueF{#1}{^{#1}}}

\makeatletter
\def\mathcenterto#1#2{\mathclap{\phantom{#1}\mathclap{#2}}\phantom{#1}}
\let\old@widetilde\widetilde
\def\widetildeto#1#2{\mathcenterto{#2}{\old@widetilde{\mathcenterto{#1}{#2\,}}}}
\makeatother

\newcommand{\leftrarrows}{\mathrel{\raise.75ex\hbox{\oalign{%
  $\scriptstyle\leftarrow$\cr
  \vrule width0pt height.5ex$\hfil\scriptstyle\relbar$\cr}}}}
\newcommand{\lrightarrows}{\mathrel{\raise.75ex\hbox{\oalign{%
  $\scriptstyle\relbar$\hfil\cr
  $\scriptstyle\vrule width0pt height.5ex\smash\rightarrow$\cr}}}}
\newcommand{\Rrelbar}{\mathrel{\raise.75ex\hbox{\oalign{%
  $\scriptstyle\relbar$\cr
  \vrule width0pt height.5ex$\scriptstyle\relbar$}}}}

\makeatletter
\def\leftrightarrowsfill@{\arrowfill@\leftrarrows\Rrelbar\lrightarrows}
\newcommand{\xleftrightarrows}[2][]{\ext@arrow 3399\leftrightarrowsfill@{#1}{#2}}
\makeatother

\newcommand{\hooklongrightarrow}{\lhook\joinrel\longrightarrow}

\DeclareRobustCommand\longtwoheadrightarrow{\relbar\joinrel\twoheadrightarrow}

\usepackage{pst-plot}
\psset{algebraic,arrows=<->}
\definecolor{light-gray}{gray}{0.95}

\makeatletter
\numberwithin{equation}{section}
\theoremstyle{plain}
  \newtheorem{thm}{\protect\theoremname}[section]
\theoremstyle{plain}
  \newtheorem*{thm*}{\protect\theoremname}
\theoremstyle{plain}
  
\theoremstyle{definition}
  \newtheorem{ex}[thm]{\protect\examplename}
  \newtheorem{defi}[thm]{\protect\definitionname}
  
\theoremstyle{plain}
  \newtheorem{prop}[thm]{\protect\propositionname}
\theoremstyle{remark}
  \newtheorem{rem}[thm]{\protect\remarkname}
    
\theoremstyle{plain}
  \newtheorem{lem}[thm]{\protect\lemmaname}
  \newtheorem{prprty}[thm]{\protect\propertyname}

\renewenvironment{proof}[1][\proofname]{\par
  \pushQED{\qed}%
  \normalfont \topsep6\p@\@plus6\p@\relax
  \trivlist
  \item[\hskip\labelsep
        \itshape
    #1\@addpunct{.---}]\ignorespaces
}{%
  \popQED\endtrivlist\@endpefalse
}

\renewcommand{\qed}{\hfill$\blacksquare$}
\makeatother

  \providecommand{\corollaryname}{Corollary}
  \providecommand{\examplename}{Example}
  \providecommand{\lemmaname}{Lemma}
  \providecommand{\propositionname}{Proposition}
  \providecommand{\remarkname}{Remark}
  \providecommand{\theoremname}{Theorem}
  \providecommand{\definitionname}{Definition}
  \providecommand{\propertyname}{Property}
  \providecommand{\notationname}{Notation}
  \providecommand{\conjecturename}{Conjecture}

\begin{document}

\begin{abstract}
The main result of this paper is the proof that all the symmetric products of a (finite) Galois-Maximal space are also Galois-Maximal spaces.
This applies to the special case of real algebraic varieties, solving the problem first stated by Biswas and D'Mello in \cite{biswas&d'mello:symmetric_products_M-curves} about symmetric products of Maximal curves, and then generalised by Baird in \cite{baird:symmetric_products_GM-curves} to Galois-Maximal curves.
We also give characterisations of these spaces and state a new definition that generalises to a larger class of spaces.
Then, we extend the characterisation in terms of the Borel cohomology given in \cite{us} to the new family.
Finally, we introduce the notion of cohomological stability and cohomological splitting, provide a systematic treatment and relate them with the properties of being a Maximal or Galois-Maximal space.
These cohomological properties play an important role in the proof of our main theorem.
 \end{abstract}


\maketitle

{\footnotesize\tableofcontents}

\section{Introduction}
\label{sec:introduction}

If $G$ is a $p$-group acting on a well-behaved space $X$, then Smith theory (\cf~\cite[Ch.~III]{bredon:introduction_compact_transformation_groups}) provides a relation between the cohomology of $X$ and of $X^G$ when considering coefficients in the field of $p$ elements $\F_p$.
In particular, for $G=C_2$, one has:
\begin{equation}
\sum_q \dim_{\F_2} H^q(X^{C_2};\F_2) \leq \sum_q \dim_{\F_2} H^q(X;\F_2).
\tag{S--T}
\label{eq:smith-thom}
\end{equation}
This inequality is known as the Smith-Thom inequality and defines \emph{Maximal spaces}: spaces for which the equality is attained.

If $\sigma\in C_2$ denotes the non-trivial element, then its pull-back $\sigma^*$ defines an action of $C_2$ on $H^*(X;\F_2)$.
This allows a refinement of the Smith--Thom inequality, sometimes referred to as the Harnack--Krasnov inequality:
\begin{equation}
\sum_q \dim_{\F_2} H^q(X^{C_2};\F_2) \leq \sum_q \dim_{\F_2} H^1(C_2, H^q(X;\F_2)).
\tag{H--K}
\label{eq:krasnov}
\end{equation}
Spaces for which the equality in this last relation is attained are referred to as \emph{Galois-Maximal spaces}.

Maximal and Galois-Maximal spaces are more common in real algebraic geometry, where the $C_2$-space is the complex locus of a real algebraic variety with the action defined by complex-conjugation.

A natural question to ask is what operations preserve these properties.
For example, it has been known for a long time that the Cartesian product of Maximal (resp.~Galois-Maximal) spaces is again a Maximal (resp.~Galois-Maximal) space.
Biswas and D'Mello considered in \cite{biswas&d'mello:symmetric_products_M-curves} the symmetric products of Maximal curves.
Soon after, Franz \cite{franz:gamma_products} addressed the general case of Maximal spaces, while Baird \cite{baird:symmetric_products_GM-curves} considered symmetric product of Galois-Maximal curves.
In all cases symmetric products preserved these properties.

In the work of this paper we deal with the general case of Galois-Maximal spaces, showing the following:

\begin{thm*}
Let $X$ be a Galois-Maximal space.
Then, $SP_nX$ is a Galois-Maximal space for every $n\in \N$.
\end{thm*}

In fact, we establish a new definition for Maximal and Galois-Maximal spaces that also applies to spaces for which the sums of \eqref{eq:smith-thom} or \eqref{eq:krasnov} are not defined.
We then extend the characterisation of these spaces in terms of Borel cohomology provided in \cite{us}.

Finally, we formulate a rigorous definition of two \emph{cohomological properties} based on some properties of the symmetric products.
Suppose given a sequence 
\begin{equation*}
X_0 \overset{f_0}{\longrightarrow} X_1 \overset{f_1}{\longrightarrow} X_2 \longrightarrow \cdots \longrightarrow X_n \overset{f_n}{\longrightarrow} X_{n+1} {\longrightarrow} \cdots
\label{eq:sequence_spaces_maps}
\end{equation*}
of spaces and maps in some category $\mathscr T$.
By a \emph{cohomological property} we mean a relation that (some of) the induced maps $f_n^*:H^*(X_{n+1})\to H^*(X_n)$ on cohomology posses with respect to some coefficients.
The inspiration for this idea comes from the results of Nakaoka \cite{nakaoka:cohomology_symmetric_products} and Steenrod \cite{steenrod:symmetric_products} about the cohomology of symmetric products.
We make the passage to $G$-spaces, focusing on the group $C_2$ and relating these properties with those of being a Maximal or Galois--Maximal space.
That will help us prove the previous theorem.

\subsection{Notation and conventions}
\label{subsec:conventions}

We use the notation $\bigcup_n X_n$ for the colimit of a tower of spaces.
More precisely, given a sequence
\[
X_0 \overset{f_0}{\longrightarrow} X_1 \overset{f_1}{\longrightarrow} X_2 \longrightarrow \cdots \longrightarrow X_n \overset{f_n}{\longrightarrow} X_{n+1} {\longrightarrow} \cdots
\]
where each map $f_n$ is a closed embedding, then
\[
\bigcup_n X_n \overset{def}{=} \varinjlim_n X_n,
\]
where the topology is the final topology with respect to the natural inclusions $j_n:X_n\to \colim_n X$: A set $U$ is open in the colimit if and only if each set $j_n^{-1}(U)$ is open in $X_n$.
When each space $X_n$ is a $G$-space, the union or colomit $X$ inherits a natural $G$-space structure, to which the inclusions $j_n:X_n\hookrightarrow X$ are equivariant.

We write $X_+$ to denote the pointed space obtained by adding an isolated point $+$ to the space $X$, $X\wedge Y$ to denote the smash product of the pointed spaces $(X,x_0)$ and $(Y,y_0)$,
\[
X\wedge Y = X\times Y\bigg/ X\vee Y,
\]
and, in the case of $G$-spaces, $X\simeq_G Y$ to indicate that there exists a homotopy equivalence between $X$ and $Y$ which is also equivariant.

Finally, we denote the Borel cohomology of the $G$-space $X$ by $H_G^*(X)$.
As it is a $\Z$-graded theory, we can consider shifted modules, that we call \emph{suspensions}, and that we denote by $\Sigma^qM$.
The reason for this is the suspension axiom of reduced Borel cohomology:
\[
\tilde{H}_G(S^p\wedge X)\cong \Sigma^p\tilde{H}_G^*(X).
\]
%
In particular, given an $\F_2$-vector space $V$, we denoted by $\Sigma^pV$ the $\Z$-graded $\F_2[z]$-module which is $V$ concentrated in degree $p$.

\section{Maximal spaces}
\label{sec:maximal_spaces}

As mentioned in the introduction, Smith theory only applies for classes of spaces that are well-behaved.
One such example is a finite \emph{$G$-CW-complex}.
These are the equivalent objects to CW-complexes in equivariant topology.
For the purpose of this work, it is enough to address only the case in which the equivariance group $G$ is finite.
We refer to \cite{peter_may:equivariant_homotopy_cohomology_theory} and \cite{matumoto:G-CW_complexes} for the general definition and their properties.

\begin{defi}
A \emph{$G$-CW-complex} is a CW-complex $X$ on which a finite group $G$ acts in such a way that the following hold:
\ \begin{enumerate}

\item The action of $G$ is cellular.

\item Invariant cells are point-wise fixed.

\end{enumerate}
A $G$-CW-complex $X$ whose underline CW-structure has finitely many cells will be called a \emph{finite $G$-CW-complex}.
The dimension of one of its top cells will be refer to as the dimension of $X$ and denoted $\dim X$.
\end{defi}

\begin{defi}
\label{def:maximal_spaces}
Let $X$ be a $C_2$-space with the equivariant homotopy type of a finite $C_2$-CW-complex.
We say that:
\begin{enumerate}

\item $X$ is a \emph{Maximal space} (also \emph{M-space}) if
\begin{equation}
\sum_{q=0}^{\dim X^{C_2}} \dim_{\F_2} H^q(X^{C_2};\F_2) = \sum_{q=0}^{\dim X} \dim_{\F_2} H^q(X;\F_2).
\label{eq:M}
\tag{M}
\end{equation}

\item $X$ is a \emph{Galois-Maximal space} (also \emph{GM-space}) if
\begin{equation}
\sum_{q=0}^{\dim X^{C_2}} \dim_{\F_2} H^q(X^{C_2};\F_2) = \sum_{q=0}^{\dim X} \dim_{\F_2} H^1(C_2, H^q(X;\F_2)).
\label{eq:GM}
\tag{GM}
\end{equation}

\end{enumerate}
A space satisfying either of these conditions will be generically called \emph{maximal} (without capital letters).
\end{defi}

\begin{rem}
\label{rem:equivariantly_formal_spaces}
Let $j$ denote the inclusion of the fibre $X\hookrightarrow X_{C_2}$.
It is a fact (\cf~\cite[III(1.18)]{tom_dieck:transformation_groups}), that being an $M$-space is equivalent to the forgetful map $j^*:H_{C_2}^*(X;\F_2) \to H^*(X;\F_2)$ being surjective.
Such spaces are sometimes referred to as \emph{totally nonhomologous to zero spaces} \cite{borel:seminar_transformation_groups} or \emph{equivariantly formal spaces} \cite{franz:gamma_products}.
\end{rem}

The term $H^1(C_2,A)$ in \eqref{eq:GM} refers to the group cohomology of $C_2$ with coefficients in $A$.
We are interested in the case when $A$ is an $\F_2$-vector space equipped with an involution $\sigma$, for which
\[
H^p(C_2,A)=A^{C_2}\Big/ \left\{a+\sigma a \mid a\in A\right\}
\]
(\cf~\cite[pp.~58--59]{brown:cohomology_of_groups}).
Hence $\dim_{\F_2} H^1(C_2,A) \leq \dim_{\F_2}A$.
In particular
\begin{equation*}
\sum_{q=0}^{n}\dim_{\F_{2}}H^{1}(C_2,H^{q}(X;\F_{2}))
 \leq \sum_{q=0}^{n}\dim_{\F_{2}}H^{q}(X;\F_{2}),
\label{eq:GM-variety_M-variety}
\end{equation*}
which means that every Maximal space is Galois-maximal.
The necessary and sufficient conditions under which the converse holds are given in the following statement. 

\begin{prop}[Krasnov, \cite{krasnov:harnack-thom_inequalities}]
\label{prop:krasnov}
A Galois-Maximal space $X$ is a Maximal space if and only if the induced action of $C_2$ on the cohomology $H^*(X;\F_2)$ is trivial.
\end{prop}

Maximal spaces are more common in real algebraic geometry, where one speaks about \emph{maximal varieties}.
If $V$ is a real variety, \ie, an algebraic variety over the field of real numbers, then its set of complex points $V(\C)$ endowed with the analytic topology and with the $C_2$-action induced from complex conjugation is a finite $C_2$-CW-complex (\cf~\cite[p.~131]{mangolte:real_algebraic_varieties}).
For non-singular, complete varieties the fixed points of this action are exactly the real points of the variety, \ie, $V(\C)^{C_2}\equiv V(\R)$.

\begin{defi}
\label{def:maximal_variety}
Let $V$ be a complete real variety.
We call $V$ a \emph{Maximal variety} or \emph{M-variety} (resp.~\emph{Galois-Maximal variety} or \emph{GM-variety}) if the $C_2$-space $V(\C)$ with the analytic topology as explained before is a Maximal space (resp.~Galois-Maximal space).
Maximal varieties of dimension one are called maximal curves, maximal varieties of dimension two are called maximal surfaces, etc.
\end{defi}

%

\begin{ex}[Maximal spaces]
\label{ex:maximal_spaces}
\ \begin{enumerate}

\item Let $\R^{p,q}$ denote the real $p$-dimensional representation of $C_2$ such that
\[
\sigma\cdot (x_1,\dots,x_p)=(x_1,\dots,x_{p-q},-x_{p-q+1},\dots,-x_p),
\]
and let $S^{p,q}=\R^{p,q}\cup\{\infty\}$ be its one-point compactification.
We called spheres of this form \emph{representation spheres}.
All representation spheres are Maximal spaces.
\label{item:representation_spheres}

\item The Cartesian product and the wedge sum of two maximal spaces, equipped with the diagonal action, are also maximal.
%
\label{item:wedge}

\item If $X$ is a topological space, the its additive induction is the $C_2$-space $L^{C_2}X\overset{def}{=}X\wedge(C_2)_+$, where $C_2$ is the free orbit and $\wedge$ is the smash product of Subsection \ref{subsec:conventions}.
Whenever $X$ has the homotopy type of a finite CW-complex, $N^{C_2}X$ is a Galois-Maximal space.
\label{item:additive_induction}

\item Similarly, there exists a \emph{multiplicative induction}, denoted by $N^{C_2}X$, and whose underlying space is the product $X\times X$, with the action of the non-trivial element $\sigma$ given by transposition:
\[
\sigma\cdot(x,x') = (x',x).
\]
Multiplicative inductions of spaces with the homotopy type of a finite CW-complex are also Galois-Maximal spaces.
\label{item:multiplicative_induction}

%

\item Any Riemann surface with at least one real point is a Galois-Maximal curve (\emph{cf.}~Krasnov \cite{krasnov:harnack-thom_inequalities}).
Also, all Abelian varieties with real points are Galois-Maximal varieties (\cf~\loccit).

\item A smooth real threefold $Y\subset \mathbf{P}_\R^4$ containing a real line $L$, such as a real cubic, is  Galois-Maximal variety (\cf~\cite{us}).

\item The blow-up of a smooth Maximal variety along a smooth real Maximal subvariety is a Maximal variety (\cf~Fu \cite{fu23}).

\item If $V$ is a Maximal (resp.~Galois-Maximal) variety and $E\to V$ is a Real vector bundle, then the projective bundle $\mathbf{P}(E)\to V$ is an Maximal (resp.~Galois-Maximal) variety (\cf~Fu \cite{fu23} for the Maximal case).

\item A curve is Maximal if and only if its Picard variety is Maximal (\cf~Biswas and D'Mello \cite{biswas&d'mello:criterion_M-curves}).

\item Non-singular complete toric varieties are Maximal (\cf~Bihan \emph{et al.}~\cite{bihan_et_al:toric_varieties_M-varieties}).

\end{enumerate}
\end{ex}

The main problem that motivated this paper is the following result of Biswas and D'Mello proven in \cite{biswas&d'mello:symmetric_products_M-curves}.

\begin{thm}[Symmetric products of M-curves]
Let $X_g$ be a Maximal curve of genus $g$.
Then $SP_nX_g$ is a Maximal variety for $n=2,3$ and $n\geq 2g-1$.
\end{thm}

The authors could not provide an answer for the cases $4<n<2g-1$, and so the problem remained open until Franz proved a much more general result.

\begin{thm}[Franz, \cite{franz:gamma_products}]
\label{thm:gamma_products}
Let $X$ be a space and let $\Gamma\subset \mathfrak{S}_n$ be a subgroup.
The $\Gamma$-product of $X$ is the quotient $X^n/\Gamma$.
For each $n$ and each $\Gamma\subset\mathfrak{S}_n$, the $\Gamma$-product of a Maximal space is another Maximal space.
In particular, setting $\Gamma=\mathfrak{S}_n$ and $X=X_g$ a Maximal curve, we see that all its symmetric products are Maximal varieties.
\end{thm}

Finally, Baird showed in \cite{baird:symmetric_products_GM-curves} that a similar theorem for the case of symmetric products of GM-curves holds

\begin{thm}[Symmetric products of GM-curves]
All the symmetric products of a Galois-Maximal curve are Galois-Maximal varieties.
\end{thm}


\section{Characterisation of maximal spaces}
\label{sec:characterisations}

In this section we provide two characterisation theorems for maximal spaces: One in terms of zero-cycles and another one using Borel cohomology.
The latter can be found in the author's paper \cite{us} as a consequence of a more general characterisation in terms of $\RO(C_2)$-graded Bredon cohomology.
It has been included here as a reference for its generalisation to infinite complexes in the next section.

The space zero-cycles associated to a given CW-complex $X$ can be understood as a generalised Eilenberg--Mac Lane space whose homotopy groups compute the reduced homology $X$.
A more precise definition is as follows:

\begin{defi}
\label{def:zero-cycle}
Given a based  space $(X,x_0)$ and an Abelian group $\Gamma$, a \emph{zero-cycle in $X$ with coefficients in $\Gamma$} is an element of the topological Abelian group with coefficients in $\Gamma$ that $X$ generates.
This spaces is defined by McCord in \cite{mccord:classifying_spaces} and denoted by $B(\Gamma,X)$.
Other notations such as $\Gamma\tilde{\otimes} X$ or $\tilde{\Gamma}[X]$ are also common.
We shall write $\Gamma[X]$ to denote the space of zero-cycles in $X$ with coefficients in $\Gamma$.
\end{defi}

The advantage of this procedure is that, in the situation where the space $X$ is a $G$-space, this action is carried out on to $\Gamma[X]$ by linearity, turning it into another $G$-space.
Then, $G$-equivariant homotopy classes of zero-cycles compute the equivariant homology of the space (\cf~\cite{pedro:equivariant_dold-thom_theorem}).
Thus, the homotopy of zero-cycles encodes all the equivariant and non-equivariant homology information of the given space and easing the relation between them.
This is very convenient in this work, because the characterisation below in terms of zero-cycles uses a result based in $\RO(C_2)$-graded Bredon cohomology; nevertheless, the condition of being a maximal space is expressed in terms of singular cohomology.
Zero-cycles are the objects that allow to make this passage from the equivariant to the non-equivariant setting easily, and vice versa.

We will be interested in the case $\Gamma=\F_2$ and $G=C_2$.
The following characterisation of maximal spaces in terms of zero-cycles will be used in the proof of the main theorem Theorem \ref{thm:main_theorem}.

\begin{thm}
\label{thm:main_characterisation}
Let $X$ be a finite $C_2$-CW-complex.
Then:
\ \begin{enumerate}

\item It is a Maximal space if and only if 
\[
\F_2[X] \simeq_{C_2} \bigoplus_{i\in I} \F_2[S^{p_i,q_i}],
\]
where $I$ is a finite set and $p_i\geq q_i\geq 0$ for all $i\in I$.

\item It is a Galois-Maximal space if and only if 
\[
\F_2[X] \simeq_{C_2} \left(\bigoplus_{i\in I} \F_2[S^{p_i,q_i}]\right)\oplus\left(\bigoplus_{j\in J} \F_2[S^{r_j}\wedge(C_2)_+]\right),
\]
where $I$ and $J$ are finite sets and $p_i\geq q_i\geq 0$ for all $i\in I$ and $r_j\geq 0$ for all $j\in J$.

\end{enumerate}
In both cases the equivariant homotopy is also a group homomorphism.
\end{thm}
\begin{proof}
The results of Theorem~6.13 and Remark~6.14 in \cite{clover_may:structure_theorem} show that every $C_2$-module over the equivariant Eilenberg--Mac Lane spectrum $H\underline{\F}_2$ must decompose as wedges of copies of $H\underline{\F}_2$ and of antipodal spheres $S_a^n$.
Then, the proof of this result can be adapted to the case of zero-cycles, as these spaces represent the spectrum.
Finally, the \eqref{eq:M} and \eqref{eq:GM} properties impose further restrictions on the splitting of $\F_2[X]$.
We suggest that the reader also consult \cite{us}.
\end{proof}

A characterisation of maximal spaces in terms of \emph{Borel cohomology} was given in \cite{us}.
If $G$ is a topological group and $X$ is a $G$-space, then the homotopy quotient or \emph{Borel construction} of $X$ is the space $X_G=(X\times EG)/G$, where $EG\to BG$ is a representative of the universal bundle of $G$.
This quotient sits in a fibration
\begin{equation}
X\hooklongrightarrow X_G \longtwoheadrightarrow BG,
\label{eq:borel_fibration}
\end{equation}
known as the \emph{Borel fibration}.
The Borel cohomology of $X$ with coefficients in a ring $R$ is the singular cohomology of the total space $X_{G}$:
\[
H_G^*(X;R) \overset{def}{=} H^*(X_{G};R).
\]
In the case of $G=C_2$, we can identify $H_{C_2}^*(pt;\F_2)$ with $\F_2[z]$, where $z\in H_{C_2}^1(pt;\F_2)$ is the generator.
Then, the cohomology of every $C_2$-space becomes a module over $\F_2[z]$.
We identify the quotient $H_{C_2}^*(pt;\F_2)/(z)$ with the field $\F_2$, regarded as an $\F_2[z]$-module.

\begin{thm}
\label{thm:characterisation_borel}
Let $X$ be a finite $C_2$-CW-complex.
Then:
\begin{enumerate}

\item It is an Maximal space if and only if $H_{C_2}^*(X;\F_2)$ is free and finitely generated over $H_{C_2}^*(pt;\F_2)$.
\label{item:free_borel_cohomology}

\item It is a Galois-Maximal space if and only if its Borel cohomology is a finite direct sum of suspensions only of the following two type of modules: the free module $H_{C_2}^*(pt;\F_2)$ and the torsion module $H_{C_2}^*(\pt;\F_2)/(z)$, $z\in H_{C_2}^1(\pt;\F_2)$.


\end{enumerate}
\end{thm}

The result for Maximal spaces was already known (\cf~\cite[III(4.16)]{tom_dieck:transformation_groups}).
In fact, as Maximal spaces can be characterised using the forgetful map $j^*$ (see Remark \ref{rem:equivariantly_formal_spaces}), the Leray--Hirsch theorem (\cf~\cite[Theorem 4D.1]{hatcher:algebraic_topology} or \cite[III(1.14)]{tom_dieck:transformation_groups}) ensures that the Borel cohomology of Maximal spaces is always free, even when the underlying $C_2$-CW-complexes are infinite (but locally finite, see the next section).


\section{Generalisation of maximal spaces}
\label{sec:generalisation}

Maximal spaces, as defined until now, must have finite cohomology, so that \eqref{eq:M} and \eqref{eq:GM} are defined.
This can be surpassed by using a characterisation of maximal spaces in terms of a spectral sequence, first due to Krasnov \cite{krasnov:harnack-thom_inequalities}.

We will consider the group $G=C_2$ and define maximal spaces using the Leray--Serre spectral sequence of the corresponding Borel fibration \eqref{eq:borel_fibration} for a space $X$. 
It thus follows that the only technical requirement on the spaces that one needs to impose is that their homotopy quotient be a Serre fibration over $BC_2$.
Nevertheless, for computational purposes, we shall restrict ourselves to spaces with the equivariant homotopy type of a \emph{locally finite} $G$-CW-complex.
This is a $G$-CW-complex $X$ such that the $n$-skeleton $X^{(n)}$ of the underlying CW-complex is finite for every $n\geq0$.
From now on, and until the end of this section, all our spaces will be assumed to have the equivariant homotopy type of a locally finite $G$-CW-complex.
Notice that a locally finite $G$-CW-complex is countable, and hence the following important result due to Milnor \cite{milnor:construction_universal_bundlesi}.

\begin{prprty}
\label{prprty:products_CW-complexes}
If $X$ and $Y$ are two locally finite $G$-CW-complexes, then the product topology and the weak topology on $X\times Y$ agree, and therefore their product has a natural $G$-CW-complex structure.
\end{prprty}

\begin{rem}
This definition of locally finite CW-complexes differs from the one provided in other texts.
For example, in \cite{geoghegan:topological_methods_group_theory} and \cite{tanaka:product_cw-complexes}, the definition of a locally finite CW-complex is equivalent to the property that the closed covering $\{\bar{e}_\alpha\}_\alpha$ consisting of the \emph{closure} of the all the cells that define the CW-complex is locally finite. 
(Our convention is that of Hatcher \cite{hatcher:algebraic_topology}: a cell is the homeomorphic image of the \emph{interior} of the disc under the attaching map; this differs from \cite{geoghegan:topological_methods_group_theory}.)
An example of a CW-complex which is locally finite in this sense but not in ours is the real line $\R$.
Conversely, the infinite wedge $W=\bigvee_{n\geq 1} S^n$ satisfies our criteria of local finiteness but the base-point $\ast\in W^{(0)}$ meets all the closed cells $\bar{e}_n=S^n$.
\end{rem}

We set the following notation:
\begin{equation}
\label{notationLS}
\parbox{\dimexpr\linewidth-4em}{%
\strut%
\itshape%
For a $C_2$-space $X$, let $\{E_*(X),d_*\}$ denote the Leray--Serre spectral of the Borel fibration $X_{C_2}\to BC_2$ with coefficients in $\F_2$.%
%
\strut
}
\tag{L--S}
\end{equation} 

\begin{defi}
\label{def:maximal_spaces_bis}
Let $X$ be a $C_2$-space with the equivariant homotopy type of a locally finite $C_2$-CW-complex.
We say that $X$ is an \emph{infinite Galois-Maximal space} if it satisfies the following two conditions:
\begin{enumerate}
   
\item There is at least one fixed point.
   
\item The Leray-Serre spectral sequence $\{E_*(X),d_*\}$ degenerates at the second page.
   
\end{enumerate}
   
We say that $X$ is an \emph{infinite Maximal space} if it satisifes the previous two conditions and, in addition, $C_2$ acts trivially on its singular cohomology $H^*(X;\F_2)$.
\end{defi}

The use of a similar\footnote{To be precise, the spectral sequence of \cite{krasnov:harnack-thom_inequalities} is used to characterise GM-varieties.
Therefore, the equivariant cohomology theory considered there is with coefficients in a $G$-sheaf, and the appropriate spectral sequence is the second Grothendieck spectral sequence (\cf~\cite[(5.2.5)]{grothendieck:tohoku_paper}).
Nevertheless, for a finite group $G$, that is equivalent to the Leray--Serre spectral sequence that arises from the Borel construction (\cf~\cite{stieglitz:equivariant_sheaf_cohomology}).} spectral sequence to characterise maximal spaces in the sense of Definition \ref{def:maximal_spaces} is considered in \cite{krasnov:harnack-thom_inequalities} already.
We have used that idea to provide a more general definition.
Of course, whenever \eqref{eq:M} and \eqref{eq:GM} are defined, both definitions are equivalent (\cf~\cite[III(1.18) and III(4.16)]{tom_dieck:transformation_groups}, \cite[Proposition 2.3]{krasnov:harnack-thom_inequalities}), so it is not necessary to make a distintion between the maximal spaces introduced here and the ones of Definition \ref{def:maximal_spaces}. 
Nevertheless, we will still use the terms \emph{finite} and \emph{infinite} sometimes to specify whether the underlying space has the homotopy type of a finite or locally finite CW-complex.

\begin{rem}
One must be careful when considering maximal spaces, as many of the results of Section~\ref{sec:maximal_spaces} might not longer be true.
For example, and as far as I know, it is unknown whether the Cartesian product of two infinite Galois-Maximal spaces is another Galois-Maximal space.
This is because the Leray--Serre spectral sequence $E_*(X\times Y)$ for the product is an Eilenberg-Moore spectral sequence (\cf~\cite[Chapters 7 and 8]{mccleary:spectral_sequences}) over a non-simply connected space, and no\footnote{%
There have been several attempts to generalise the conditions under which the Eilenberg--Moore spectral sequence converges. We highlight Larry Smith \cite{larry_smith:kunneth_theorem}, Hodgkin \cite{hodgkin:equivariant_kunneth_theorem_K-theory} and Greenless \cite{greenless:generalized_eilenberg-moore}.
In fact, with the work of Larry Smith, and using the generalisation of Theorem \ref{thm:characterisation_borel}, it is possible to show that the Cartesian product of an infinite Maximal space with an infinite Galois-Maximal space is Galois-Maximal, and that the product of two infinite Maximal spaces is again Maximal.
}
results for the convergence are known.
\end{rem}

To conclude this section, we would like to extend the characterisation of Theorem \ref{thm:characterisation_borel} so that it includes infinite maximal spaces as in Definition \ref{def:maximal_spaces_bis}.
We start by noting that the Leray--Serre spectral sequence associated to the Borel fibration \eqref{eq:borel_fibration} computes the Borel cohomology of the space.
In the case of maximal spaces, this spectral sequence degenerates at the second page and, as we are considering $\F_2$ coefficients, there is no extension problem.
Thus
\begin{equation}
H_{C_2}^n(X;\F_2) \cong \bigoplus_{p+q= n} E_2^{p,q} = \bigoplus_{p+q= n} H^p(C_2,H^q(X;\F_2))
\label{eq:decomposition_borel}
\end{equation}
as Abelian groups.
This is in fact an isomorphism of $\F_2[z]$-modules, as the Leray--Serre spectral sequence has a module structure over $H_{C_2}^*(pt;\F_2)=\F_2[z]$ inherited by the $E_2$-term.
Actually, all the $E_r$-terms contain a copy of $\F_2[z]$ that comes from the fixed points, that appears first in $E_2^{*,0}$ and that survive through all the stages (infinite cocycles).
Moreover, the differentials are module homomorphisms.
These facts, all together, lead to the following property.

\begin{prprty}
\label{prprty:spectral_sequence}
Let $a\in E_r^{p,q}$ be non-trivial.
Then, there exists $a_0\in E_r^{0,q}$ such that $a=z^p a_0$, and it follows that $d_ra=0$ if and only if $d_ra_0=0$.
\end{prprty}

When proving the more general version of Theorem \ref{thm:characterisation_borel}, the crucial observation is that the page at which the spectral sequence degenerates determines the types of modules that can appear in the decomposition \eqref{eq:decomposition_borel}.
They can be worked out by looking at the \emph{torsion order} of the Borel cohomology.
By this, we mean the following: Let $M$ be a graded module over $H_{C_2}^*(pt;\F_2)$ and let $m\in M$ be a homogeneous torsion element (non-trivial); the \emph{torsion order} of $m$ is the integer $r$ such that $z^rm =0$ but $z^{r-1}m\neq 0$.
We will refer to $m$ as a torsion element of \emph{order $r$}.
The module $\F_2[z]/(z^r)$ contains elements of torsion orders $1,2,\dots, r$.

\begin{prop}
Let $X$ be a locally finite $C_2$-CW-complex.
If the Leray--Serre spectral sequence for the Borel fibration of $X$ does not degenerate at the second page, then the Borel cohomology of $X$ must have torsion elements of order strictly bigger than one.
\end{prop}
\begin{proof}
Let $r_0>2$ be the least integer such that $d_{r_0}\not\equiv 0$.
As Property \ref{prprty:spectral_sequence} shows, it is enough to consider the restriction of $d_{r_0}$ to the zeroth column $E_{r_0}^{0,*}$.
Let
\[
q_0 = \min \left \{q \;\middle |\; d_{r_0}:E_{r_0}^{0,q}\to E_{r_0}^{r_0,q-r_0+1} \mbox{ is non-trivial}\right \}.
\]
Due to dimensional reasons, $q_0\geq r_0-1$ and, in fact, $q_0>r_0-1$, for $q_0=r_0-1$ contradicts the fact that $E_2^{*,0}=E_\infty^{*,0}$.
Regardless of its value, there will be an element $a_1\in E_{r_0}^{0,q_0-r+1}$ such that $z^{r_0}a_1=d_{r_0}a_0\equiv 0$ in $E_{r_0+1}$.
If $a_1$ is an infinite cocycle, then we are done.
Otherwise there exists $r_1>r_0$ such that $d_{r_1}a_1\neq0 $.
Let $a_2\in E_{r_1}^{0,q_0-r_0+1}$ be such that $d_{r_1}a_1 = z^{r_1}a_2$.
Once again, either it is an infinite cocycle and the proof is completed, or it is not and there exists $r_2>r_1$ such that $d_{r_2}a_2\neq 0$.

As $q_0$ is finite, this must stop eventually: after $n$ steps we will get an element
\[
a_n\in E_{r_n}^{0,q_0+n-r_0-\dots-r_{n-1}}
\]
with $d_{r_{n}}a_{n} =0$ because of dimensional reasons.
That element will have torsion order $r_n>1$.

The proof becomes clearer with the figure below.
\end{proof}

\begin{figure}[H]
     \centering
\begin{pspicture}(-6,-4.5)(6,4.1)
\psset{unit = 0.7}
\psline[linecolor=gray]{-}(-5.5,-5)(-2.5,-5)
\psline[linecolor=gray]{->}(-1.5,-5)(5,-5)
\psline[linecolor=gray]{-}(-5,-5.5)(-5,0.5)
\psline[linecolor=gray]{->}(-5,1.5)(-5,5)
\rput[b]{0}(-2,-5){$\dots$}
\rput[B]{0}(-5,0.75){$\vdots$}
\rput[Bl]{0}(5.1,-5){$\textcolor{gray}{p}$}
\rput[b]{0}(-5,5.1){$\textcolor{gray}{q}$}

\rput[b]{0}(5,4.5){$E_*^{p,q}$}

\psline{->}(-5,3.75)(-3.5,3.25)
\pscircle[fillstyle=solid,fillcolor=black,dimen=inner](-5,3.75){0.03}
\rput[Br]{0}(-5.25,3.75){$a_0$}
\psline[linestyle=dashed,dash=3pt 2pt]{-}(-3.5,-5.1)(-3.5,3.25)

\psline{->}(-5,3.25)(-2.75,2.5)
\pscircle[fillstyle=solid,fillcolor=black,dimen=inner](-5,3.25){0.03}
\rput[br]{0}(-5.25,3.25){$a_1$}
\psline[linestyle=dotted,dotsep=2pt]{|->>}(-5.014,3.25)(-3.5,3.25)
\rput[bl]{0}(-3.4,3.25){$d_{r_0}a_0=z^{r_0}a_1$}
\psline[linestyle=dashed,dash=3pt 2pt]{-}(-2.75,-5.1)(-2.75,2.5)
\rput[t]{0}(-3.5,-5.25){$r_0$}

\pscircle[fillstyle=solid,fillcolor=black,dimen=inner](-5,2.5){0.03}
\rput[Br]{0}(-5.25,2.5){$a_2$}
\psline[linestyle=dotted,dotsep=2pt]{|->>}(-5.014,2.5)(-2.75,2.5)
\rput[Bl]{0}(-2.65,2.5){$d_{r_1}a_1=z^{r_1}a_2$}
\rput[t]{0}(-2.75,-5.25){$r_1$}

\psline{->}(-5,-0.5)(-1.25,-1.75)
\pscircle[fillstyle=solid,fillcolor=black,dimen=inner](-5,-0.5){0.03}
\rput[Br]{0}(-5.25,-0.5){$a_{n-2}$}
\psline[linestyle=dashed,dash=3pt 2pt]{-}(-1.25,-5.1)(-1.25,-1.75)
\rput[t]{0}(-1.25,-5.25){$r_{n-2}$}

\psline{->}(-5,-1.75)(1,-3.75)
\pscircle[fillstyle=solid,fillcolor=black,dimen=inner](-5,-1.75){0.03}
\rput[Br]{0}(-5.25,-1.75){$a_{n-1}$}
\psline[linestyle=dotted,dotsep=2pt]{|->>}(-5.014,-1.75)(-1.25,-1.75)
\rput[Bl]{0}(-1.15,-1.75){$d_{r_{n-2}}a_{n-2}=z^{r_{n-2}}a_{n-1}$}
\psline[linestyle=dashed,dash=3pt 2pt]{-}(1,-5.1)(1,-3.75)
\rput[t]{0}(1,-5.25){$r_{n-1}$}

\psline{->}(-5,-3.75)(2.5,-6.25)
\pscircle[fillstyle=solid,fillcolor=black,dimen=inner](-5,-3.75){0.03}
\rput[Br]{0}(-5.25,-3.75){$a_n$}
\psline[linestyle=dotted,dotsep=2pt]{|->>}(-5.014,-3.75)(1,-3.75)
\rput[Bl]{0}(1.1,-3.75){$d_{r_{n-1}}a_{n-1}=z^{r_{n-1}}a_n$}
\rput[Bl]{0}(2.6,-6.25){$d_{r_{n}}a_{n}$}


\end{pspicture}
         \caption{Illustration of the torsion phenomenon due to the existence of $r>2$ such that $d_r$ is non-trivial.}
         \label{fig:M_2shifted}
\end{figure}
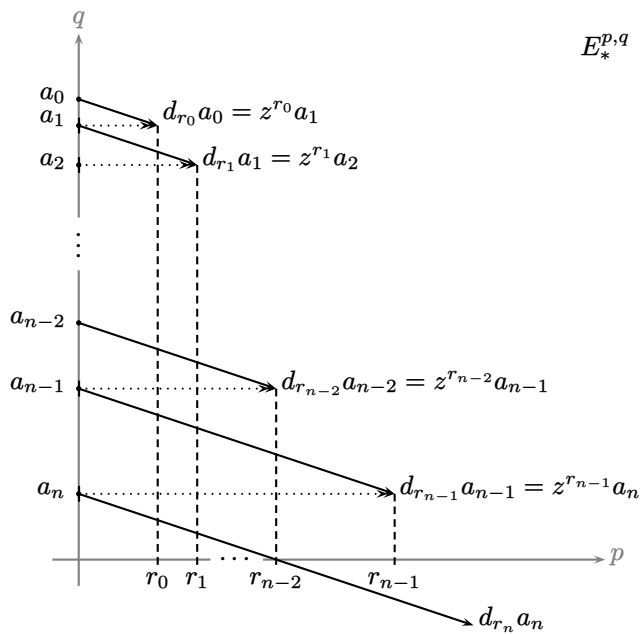


\section{Cohomological properties}
\label{sec:cohomological_properties}

In this section we give a systematic study of the two cohomological properties satisfied by the symmetric products and which play an important role in the proof of Theorem \ref{thm:main_theorem}.
One of these properties is known as \emph{cohomological stability}; it is well known in group theory and some results about it for topological spaces can be found in the paper of Jiménez and Wilson \cite{jimenez&wilson:homological_stability}.
The results of Steenrod \cite{steenrod:symmetric_products} about the cohomology of the symmetric products constitute the archetypal example.
The other cohomological feature, first suggested by the work of Nakaoka in \cite{nakaoka:cohomology_symmetric_products} also for the cohomology symmetric products, has been tentatively called \emph{cohomological splitting}.

We need some preliminary results on the inverse limits of groups.
The reader is refer to Section 3.5 in Weibel \cite{weibel:homological_algebra}, Chapter 7 in Switzer \cite{switzer:algebraic_topology} and additional topic 3F in Hatcher \cite{hatcher:algebraic_topology} for further details.

A \emph{tower of groups} \cite[p.~80]{weibel:homological_algebra} is a family of groups and homomorphisms of the form
\[
\cdots \longrightarrow A_{n+1}\longrightarrow A_n \longrightarrow \cdots \longrightarrow A_2 \longrightarrow A_1 \longrightarrow A_0 .
\]
We will denote it by $\{A_n\}$.

\begin{defi}
A tower $\{A_n\}$ of Abelian groups satisfies the \emph{Mittag-Leffler condition} (\cf~\cite[Definition 3.5.6]{weibel:homological_algebra}) if for each $m$, there exists some $n_0 >m$ such that the image of $A_n\to A_m$ equals the image of $A_{n_0}\to A_m$ for all $n > n_0$.
 (The images of the different $A_n$ in $A_m$ satisfy the descending chain condition.)
\end{defi}

The Mittag-Leffler condition is satisfied if, for example, all the maps $A_{n+1}\to A_n$ in the tower $\{A_n\}$ are onto.

\begin{prop}[See {\cite[Proposition 3.5.7]{weibel:homological_algebra}}]
We denote by $\varprojlim^1$ the first right-derived functor of $\varprojlim$.
If a tower $\{A_n\}$ of Abelian groups satisfies the Mittag-Leffler condition, then
\[
\varprojlim^1 A_n = 0.
\]
\end{prop}

\begin{defi}
Let 
\begin{equation}
X_0 \overset{f_0}{\longrightarrow} X_1 \overset{f_1}{\longrightarrow} X_2 \longrightarrow \cdots \longrightarrow X_n \overset{f_n}{\longrightarrow} X_{n+1} {\longrightarrow} \cdots
\label{eq:sequence}
\end{equation}
be a sequence of spaces and maps in some appropriate category $\mathscr T$.
We say that the sequence \eqref{eq:sequence} exhibits:
\begin{enumerate}

\item \emph{Cohomological stability} with respect to the ring $R$ if, for each $q>0$, there exists $n_0\in\N$ (depending on $q$) such that the maps 
\[
f_n^*:H^q(X_{n+1};R) \to H^q(X_n;R)
\]
are all isomorphisms for $n>n_0$.

\item The \emph{cohomological splitting property} with respect to the ring $R$ if each morphism
\[
f_n^*:H^*(X_{n+1};R) \to H^*(X_n;R)
\]
admits a section.

\end{enumerate}
In the case when $\mathscr T$ is a $G$-category, the sections of the cohomological splitting property must be equivariant with respect to the induced action of $G$ on cohomology. 
\end{defi}

\begin{prprty}
\label{prpty:products_cohomological_properties}
Products, wedge sums and additive and multiplicative inductions of sequences (performed term-wise) preserve cohomological stability and the cohomological splitting property (some of these results might require to consider coefficients in a field).
\end{prprty}

From now on we shall restrict ourselves to sequences whose spaces are finite pointed $C_2$-CW-complexes, and whose morphisms are equivariant, base-preserving, cellular maps, although some of the results might hold in a more general setting.
We will further assume that each map $f_n:X_n\to X_{n+1}$ is a closed embedding, and that the quotient
\[
X_n/X_{n-1} = \bigvee_i S^{\alpha_i}
\]
satisfies $\alpha_i>\dim X_{n-1}$ for all $i$ (this is to ensure that the union is a locally finite $C_2$-CW-complex).
Finally, we will consider cohomological properties with respect to the field $\F_2$ and write $H^*(X)$ instead of $H^*(X;\F_2)$.

\begin{lem}
\label{lem:cohomological_properties}
Let 
\[
X_0 \overset{f_0}{\longrightarrow} X_1 \overset{f_1}{\longrightarrow} X_2 \longrightarrow \cdots \longrightarrow X_n \overset{f_n}{\longrightarrow} X_{n+1} {\longrightarrow} \cdots
\]
be a sequence as above and let
\[
X=\bigcup_n X_n.
\]
Denote by $j_n:X_n\\hookrightarrow X$ the natural inclusion of $X_n$ into $X$.
\begin{enumerate}

\item If the sequence exhibits cohomological stability, then:
\begin{enumerate}

\item For each $q>0$, there exists $n_0$ such that $j_n^*:H^q(X) \to H^q(X_n)$ is an isomorphism for all $n>n_0$.
In particular the inclusions $j_n$ induce an isomorphism
\[
H^*(X) \cong \varprojlim_n H^*(X_n).
\]

\item For each $q>0$, there exists a space $X_n$ such that $j_n^*:H^r(X)\to H^r(X_n)$ is an isomorphism for all $r\leq q$.
\label{item:cohomological_stability_b}
\end{enumerate}

\item If, instead, the sequence possesses the cohomological splitting property, then:
\begin{enumerate}

\item There is an isomorphism
\[
H^*(X) \cong \varprojlim_n H^*(X_n)
\]
\label{item:1}
induced by the inclusions $j_n:X_n\to X$.

\item For each $n$, the morphism $j_n^*:H^*(X)\to H^*(X_n)$ admits a section, which is equivariant.
\label{item:section}

\end{enumerate}
 
\end{enumerate}
\end{lem}
\begin{proof}
We first explain how the isomorphism $H^*(X) \cong \varprojlim_n H^*(X_n)$ is obtained.
The Milnor exact sequence (\cf~\cite[Theorem 3F.8]{hatcher:algebraic_topology} or \cite[7.66]{switzer:algebraic_topology}) yields
\begin{equation}
0 \longrightarrow \varprojlim_n^1 H^{q-1}(X_n) \longrightarrow H^q(X) \longrightarrow \varprojlim_n H^q(X_n) \longrightarrow 0 .
\label{eq:milnor_exact_sequence}
\end{equation}
Moreover, either of the cohomological properties implies that the Mittag-Leffler condition is satisfied; therefore
\[
\varprojlim_n^1 H^{q-1}(X_n) \equiv 0,
\]
%
as desired.
It is a fact (\cf~\cite[7.66 Proposition]{switzer:algebraic_topology}) that this isomorphism is induced by the inclusions $j_n:X_n\to X$.
Finally, as these maps is equivariant, the above is an isomorphism of $C_2$-groups.

In the case of cohomological stability, the inverse limit stabilises for each $q$, meaning that
\[
\varprojlim_n H^q(X_n) \cong H^q(X_{n_0})
\]
for some $n_0\gg0$.
Now statement (\ref{item:cohomological_stability_b}) follows from this by considering the space $X_n$, where
\[
n = \max_{0\leq r\leq q} \left\{m \;\middle|\; j_n^*:H^r(X)\overset{\cong}{\to} H^r(X_m) \right\}.
\]

Regarding the section $s_n:H^*(X_n)\to H^*(X)$ of statement \ref{item:section}, it is defined as follows: For each $n\geq 0$, let $t_n$ denote the section of $f_n^*$; then the image of an element $a\in H^*(X_n)$ under $s_n$ is the element
\[
(a_m)_m = \begin{cases}
f_{n,m}^*(a), & \mbox{if } m\leq n, \\
a, & \mbox{if } m=n, \\
t_{n,m}(a), & \mbox{if } m>n.
\end{cases}
\]
Here $f_{n,m}^*$ denotes the composite morphism $f_n^*f_{n-1}^*\cdots f_m^*$.
Similarly, $t_{n,m} = t_m t_{m+1} \cdots t_n$.
It is a straightforward exercise to verify that the image of each $s_n$ is indeed contained in the inverse limit of the $H^*(X_n)$.
\end{proof}


The main result of this section is the following.

\begin{thm}
\label{thm:cohomological_properties}
Let
\[
X_0 \overset{f_0}{\longrightarrow} X_1 \overset{f_1}{\longrightarrow} X_2 \longrightarrow \cdots \longrightarrow X_n \overset{f_n}{\longrightarrow} X_{n+1} {\longrightarrow} \cdots
\]
be as above and let
\[
X=\bigcup_n X_n.
\]
\begin{enumerate}

\item If the sequence has cohomological stability and each $X_n$ is maximal, then $X$ is also a maximal space.
\label{item:cohomological_stability}

\item Conversely, assume that $X_n^{C_2}$ is non-empty for every $n$.
Then, if the sequence possesses the cohomological splitting property and $X$ is maximal, all the $X_n$ are also maximal spaces.
\label{item:cohomological_splitting}

\end{enumerate}
\end{thm}
\begin{proof}
We will consider first the \eqref{eq:GM} property and show that the corresponding Leray--Serray spectral sequences for each $X_n$ characterising this property degenerate at the second page.
Recall that for a $C_2$-space $X$, we denote this spectral sequence by $E_*(X)$ (see notation \eqref{notationLS}).
For the results about Maximal spaces we will make use of Lemma \ref{lem:cohomological_properties} accordingly.



We start with (1).
Given $p$ and $q$, we will show that $E_2^{p,q}=E_\infty^{p,q}(X)$, for which we will prove that all the relevant differentials $d_r^{p,q}:E_r^{p,q}(X)\to E_r^{p+r,q-r+1}$ and $d_r^{p-r,q+r-1}:E_r^{p-r,q+r-1}(X)\to E_r^{p,q}$ vanish for $r\geq 2$.
In order to do so, we will use the morphism of spectral sequences induced by the natural inclusion $j_n:X_n\hookrightarrow X$, for some appropriate $n$.
We denote this morphism $\varphi_n:E_*(X)\to E_*(X_n)$.
At the second page, $\varphi_n$ is given by
\[
\varphi_n=H^p(C_2,-)j_n^* :E_2^{p,q}(X) = H^p(C_2,H^q(X;\F_2)) \longrightarrow E_2^{p,q}(X_n) = H^p(C_2,H^q(X;\F_2)),
\]
and due to cohomological stability we can choose $n$ such that $j_n^*:H^q(X;\F_2) \to H^q(X_n;\F_2)$ becomes an isomorphism (see the first part of Lemma \ref{lem:cohomological_properties}).
This provides linear maps $\psi_n^{p,q}:E_2^{p,q}(X)\to E_2^{p,q}(X_n)$ which are inverses for the $\varphi_n^{p,q}$.
It is extremely important to notice that these inverses occur at the level of $\F_2$-vector spaces; the map $\{\psi_n^{p,q}\}$ might not commute with the differentials, and therefore does not define a morphism of spectral sequences.
In particular, this does not make $\varphi_n$ an isomorphism of spectral sequences.

Our choice for $n$ is such that $j_n^*:H^k(X)\to H^k(X_n)$ is an isomorphisms for all $k\leq p^pq^q$.
Let $d_2\equiv d_2^{p,q}$ and let $d_2^n$ denote the differential $E_2^{p,q}(X_n)\to E_2^{p+2,q-1}(X_n)$.
Given $a\in E_2^{p,q}(X)$ we have:
\[
d_2a = \psi^{p+2,q-1}(d_2^n\varphi_n^{p,q}(a)) =0,
\]
as the space $X_n$ is maximal.

Not only that, but due to the choice of $n$, the same is true for all $a\in E_2^{s,t}$ with $s>0$ and $t\leq p^pq^q$.
In particular, $d_2^{p-2,q+1}:E_2^{p-2,q+1}(X)\to E_2^{p,q}(X)$ also vanishes, so that $E_3^{p,q}(X)\equiv E_2^{p,q}(X)$.
Therefore, the $\psi_n^{p,q}$ are still defined, so each $\varphi_n^{p,q}$ is a linear isomorphism.
Hence, we can repeat the same strategy: For $a\in E_3^{p,q}(X)$, we have
\[
d_3a = \psi^{p+3,q-2}(d_3^n\varphi_n^{p,q}(a)) =0,
\]
and the choice of $n$ also ensures that $d_3^{p-3,q+2}:E_3^{p-3,q+2}(X)\to E_r^{p,q}(X)$ is trivial.
Actually, our choice guarantees that those differentials $d_r^{s,t}$ which can hit $E_r^{p,q}$ will be trivial as long as $s\geq 0$, and for those values of $t$ over which $n$ has no influence, we will have $s<0$, so that $E_r^{p,q}(X)$ will never be hit by any non-zero differential.
Similarly, all the differentials $d_r^{p,q}$ must also be trivial, proving that $E_2^{p,q}(X)=E_\infty^{p,q}(X)$.

Finally, in the case where all the spaces are Maximal spaces, so will be $X$.
Indeed, Lemma \ref{lem:cohomological_properties}(\ref{item:cohomological_stability_b}) implies that, for each $q$, we can choose a space $X_n$ such that $H^q(X)\cong H^q(X_n)$ as $C_2$ groups; as $X_n$ is Maximal, the action of $C_2$ on cohomology must be trivial, and by varying $q$ we can show that $C_2$ acts trivially on $H^*(X)$, \ie, $X$ is a Maximal space.

The proof for (2) follows a similar strategy. 
The cohomological splitting property guarantees the existence of a section $s_n$ to each $j_n^*:H^*(X)\to H^*(X_n)$ (second part of Lemma \ref{lem:cohomological_properties}).
In particular, each morphism $\varphi_n$ is onto at the $E_2$-term. 
By this, we mean that each linear map $\varphi_n^{p,q}:E_2^{p,q}(X)\to E_2^{p,q}(X_n)$ is surjective, with a section given by $\sigma_n^{p,q}=H^p(C_2,-)s_n$.

This situation is enough to prove that any differential $d_2^n$ of $E_2(X_n)$ vanishes.
Indeed, given $a_n\in E_2^{p,q}(X_n)$, let $a\in E_2^{p,q}(X)$ such that $a_n=\varphi_n(a)$.
Then:
\begin{equation}
d_2^na_n= d_2^n(\varphi_na) = \varphi_n(d_na) = 0,
\label{eq:trivial_differential}
\end{equation}
as $X$ is maximal by hypothesis.

This proves that $E_3(X_n)\equiv E_2(X_n)$, and we can start the same process again: the $\sigma_n^{p,q}$ are still well-defined, which makes each $\varphi_n^{p,q}:E_3^{p,q}(X)\to E_r^{p+3,q-2}(X)$ onto; then we can proceed as in \eqref{eq:trivial_differential} to show that all the differentials $d_3^{p,q}:E_3^{p,q}(X_n) \to E_3^{p+3,q-2}(X_n)$ are trivial.
The result is that $d_r^{p,q}\equiv 0$ for all $p$ and $q$ and for all $r\geq 2$.

Finally, if $X$ were a Maximal space, then the action of $C_2$ on the cohomology of each $X_n$ would also be trivial, for given $a_n\in H^*(X_n)$, it is enough to consider $a\in H^*(X)$ with $a_n=j_n^*a$, and
\[
\sigma^*a_n = \sigma^*(j_n^*a) = j_n^*(\sigma^*a)=j_n^*a = a. 
\]
\end{proof}
 
\section{Properties of the symmetric products}
\label{sec:symmetric_products}

In this section we review some of the most important properties that the symmetric products posses and that will be necessary for the proof of the main theorem.
Recall that the $n$-fold symmetric product of a space $X$, $n\geq 2$, is the orbit space $X^n/\mathfrak{S}_n$, where $\mathfrak{S}_n$ denotes the symmetric group in $n$-letters and acts on $X^n$ in the obvious way:
\[
\pi\cdot (x_1,\dots,x_n)=(x_{\pi(1)},\dots,x_{\pi(n)}), \qquad \pi\in \mathfrak{S}_n.
\]
We denote the equivalence class of the point $(x_1,\dots,x_n)$ as $\{x_1,\dots,x_n\}$.
The convention is that $SP_0X$ is a singleton and $SP_1X\equiv X$.

We endow each $SP_nX$ with the quotient topology with respect to the quotient map $\mu_n:X^n\to SP_nX$.
When $X$ is a CW-complex with countably many cells, there exists a CW-complex structure on $SP_nX$ (\cf~\cite[Note 5.2.2]{aguilar_et_al:algebraic_topology_homotopical_viewpoint}).
Moreover, in the case of based CW-complexes, each $n$-fold symmetric product becomes a pointed space too, and then $SP_nX$ lies in $SP_{n+1}X$ canonically.
This inclusion, that we denote by $i_{n,n+1}$, is a closed embedding, so that $SP_nX$ can be regarded as subcomplex of $SP_{n+1}X$.
Finally, in this case, we define the infinite symmetric product $SP_\infty X$ as the union or colimit of all the finite symmetric products:
\[
SP_\infty X = \bigcup_n SP_n X
\]
with the union topology, as explained in Subsection~\ref{subsec:conventions}.
There are inclusions $i_n:SP_nX\hookrightarrow SP_\infty X$, for every $n$.

The above transcripts, \emph{mutatis mutandis}, to the case where $X$ is a countable based $G$-CW-complex (with $G$ a finite group). 
In this case, the inclusions $i_{n,n+1}:SP_nX\hookrightarrow SP_{n+1}X$ and $i_n:SP_nX \hookrightarrow SP_{n+1}X$ become equivariant.
The next results summarises the cohomological properties of the symmetric products of a finite, based $C_2$-CW-complex.

\begin{lem}
\label{lem:symmetric_products}
Let $(X, x_0)$ be a finite pointed $C_2$-CW-complex.
Then:
\begin{enumerate}

\item The inclusion $i_{n,n+1}$ is given by the assignment
\[
\{x_1,\dots,x_n\}\longmapsto \{x_0,x_1,\dots,x_n\}.
\]

\item The sequence
\[
SP_0 X \overset{i_{0,1}}{\longrightarrow} SP_1X \overset{i_{1,2}}{\longrightarrow} SP_2X  \longrightarrow \cdots \longrightarrow SP_nX \overset{i_{n,n+1}}{\longrightarrow} SP_{n+1}X \longrightarrow \cdots
\]
is a sequence in the category of finite, based $C_2$-CW-complexes and equivariant, base-preserving, cellular maps, and satisfies all the conditions imposed in Section \ref{sec:cohomological_properties}.
It also exhibits \emph{both} cohomological stability \emph{and} the cohomological splitting property.

\end{enumerate}
\end{lem}
\begin{proof}
The cohomological splitting property was shown in Nakaoka \cite{nakaoka:cohomology_symmetric_products} for simplicial complexes, and a similar reasoning works for $G$-CW-complexes ($G$ any finite group).

Regarding cohomological stability, it is a classical result of Steenrod \cite{steenrod:symmetric_products} that
\[
H^q(SP_nX;\F_2) \cong H^q(SP_\infty X;\F_2)
\]
(among other coeffcients) when $q<n$.
The results of Nakaoka imply that this isomorphism is induced by the inclusion $i_n:SP_nX\rightarrow SP_\infty X$, and therefore it also holds in the case of $G$-spaces.
\end{proof}

Another important result about symmetric products is that they preserve homology, meaning that if $H_q(X)\cong H_q(Y)$, then
\[
H_q(SP_nX)\cong H_q(SP_nY)
\]
(\cf~\cite[Corollary 5.2.10]{aguilar_et_al:algebraic_topology_homotopical_viewpoint}).
This is often abbreviated by saying that the symmetric products are homological functors.
A similar result holds for zero-cycles.

\begin{lem}
\label{lem:functor_0-cycles}
The functors $SP_n$, $n\in\N$, are functors of zero-cycles.
More precisely: If $X$ and $Y$ are $C_2$-spaces and $f:\F_2[X]\to\F_2[Y]$ is an equivariant homotopy equivalence which is also a group homomorphism, then for each $n\in \N$, there exists an equivariant homotopy equivalence $\tilde{f}_n:\F_2[SP_n X] \to \F_2[SP_nY]$ which is also a group homomorphism.
\end{lem}
\begin{proof}
The map $\tilde{f}_n:\F_2[SP_nX]\to \F_2[SP_nY]$ is defined as the linear extension of the long composite $SP_nX\to{\F}_2[SP_nY]$ in the diagram below:
\[
\begin{tikzcd}[column sep=huge,row sep=huge]
 & SP_n {\F}_2[Y] \arrow{r}{\epsilon_Y} & {\F}_2[SP_nY] \\
SP_nX \arrow[hook]{r} & SP_n {\F}_2[X] \arrow{u}{SP_n f} \arrow{r}[swap]{\epsilon_X} & {\F}_2[SP_n X],
\end{tikzcd}
\]
There, for a space $Z$, $\epsilon_Z:SP_n{\F}_2[Z]\to{\F}_2[SP_nZ]$ is a continuous map defined by the following two properties:
\begin{enumerate}

\item For $z_1,\dots,z_n\in Z$,
\[
\epsilon_Z\{z_1,\dots,z_n\}= \{z_1,\dots,z_n\}.
\]

\item It is multi-linear, \ie
\[
\epsilon_Z\{z_1,\dots,z_i+z_i',\dots,z_n\}= \epsilon_Z\{z_1,\dots,z_i,\dots,z_n\} + \epsilon_Z\{z_1,\dots,z_i',\dots,z_n\}.
\]

\end{enumerate}

We suggest that the reader check that this \guillemetleft tilde\guillemetright\ ($\sim$) construction defines indeed a continuous map and that it is functorial, so that if $g:\F_2[Y]\to \F_2[X]$ is a homotopy inverse for $f$ and $h:\F_2[X]\times I\to \F_2[Y]$ is the corresponding homotopy, then $\tilde{g}_n$ is a homotopy inverse for $\tilde{f}_n$ and $\tilde{h}$ is a valid homotopy given by
\[
\tilde{h}_t = \widetildeto{\mathbf{W}}{h_t}.
\]
\end{proof}

\begin{prprty}[Aguilar \etal, {\cite{aguilar_et_al:algebraic_topology_homotopical_viewpoint}}]
\label{prpty:symmetric_product_multiplicative_induction}
Let $(X,x_0)$ be a pointed space.
Then
\[
SP_\infty \left(L^{C_2}X\right) \cong N^{C_2} SP_\infty X,
\]
as $C_2$-spaces, where $L^{C_2}$ and $N^{C_2}$ denote the additive and multiplicative inductions introduced in Example \ref{ex:maximal_spaces}.
\end{prprty}

\begin{rem}
\label{rem:limits_colimits}
The previous property makes an identification between the limit of a product and a product of the limit.
This is possible due to the assumption of locally finiteness for the complexes, so that the product is automatically a $C_2$-CW-complex (\cf~Property \ref{prprty:products_CW-complexes}), and the following categorical fact:
\begin{quote}
\itshape
In the category of topological spaces, finite limits commute with filtered colimits.
In particular, if $\{(X_n,f_n)\}_n$ and $\{(Y_n,g_n)\}_n$ are two family of topological spaces and maps such that all the maps $f_n:X_n\to X_{n+1}$ and $g_n:Y_n\to Y_{n+1}$ are closed embeddings, then 
\[
\left( \bigcup_n X_n \right) \times \left( \bigcup_n Y_n\right) \cong \bigcup_n (X_n\times Y_n).
\]
\end{quote}
Indeed, it is well known (\cf~\cite[IX.2]{maclane:category_theory}) that finite limits commute with filtered colimits when the category is cocomplete.
And the category of topological spaces and continuous maps is an example of a cocomplete category (\cf~\cite[12.6 Examples]{adamek_et_al:abstract_and_concrete_categories}).
We will use this result again in the proof of Theorem \ref{thm:main_theorem}.
\end{rem}

\section{Main result}
\label{sec:main_result}

\begin{thm}
\label{thm:main_theorem}
Let $X$ be a finite Galois-Maximal space.
Then, $SP_nX$ is a Galois-Maximal space for every $n\leq \infty$.
\end{thm}

\begin{proof}
We start by proving the result for $C_2$-spaces of the form
\begin{equation}
W=\left(\bigvee_{i\in I} S^{p_i,q_i} \right) \vee \left(\bigvee_{j\in J} [S^{r_j}\wedge(C_2)_+]\right),
\label{eq:W}
\end{equation}
(recall that $I$ and $J$ are finite sets).
Once this is solved, we will explain how to address the general case of any finite space.


This space $W$ is a Galois-Maximal space, as follows from (\ref{item:representation_spheres})--(\ref{item:wedge}) and (\ref{item:additive_induction}) in Example \ref{ex:maximal_spaces}.
The aim is to show that $SP_\infty W$ is a Galois-Maximal space too and then use the results about cohomological stability and cohomological splitting of Theorem \ref{thm:cohomological_properties} for the particular case of symmetric products, as in Lemma \ref{lem:symmetric_products}, to obtain the result for each finite symmetric power $SP_nW$.

First, notice that
\begin{equation*}
SP_\infty W \cong \left( \prod_{i\in I} SP_\infty S^{p_i,q_i} \right)\times \left( \prod_{j\in J} N_{C_2}SP_\infty S^{r_j}\right) \cong \bigcup_n\left\{\left(\prod_{i\in I} SP_n S^{p_i,q_i}\right)\times \left(\prod_{j\in J} N^{C_2} SP_n S^{r_j} \right)\right\} 
\end{equation*}
due to Property \ref{prpty:symmetric_product_multiplicative_induction} (we are using the commutativity properties of products and colimits again as explained in Remark \ref{rem:limits_colimits}).
Each $SP_n S^{p_i,q_i}$ is a Maximal space by the results of Franz given in Theorem \ref{thm:gamma_products}, and so is their product; similarly, $N^{C_2}SP_nS^{r_j}$ is a Galois-Maximal space, for all $n$ and $j$, as mentioned in Example \ref{ex:maximal_spaces}(\ref{item:multiplicative_induction}).
Hence, each product
\[
Z_n=\left(\prod_{i\in I} SP_n S^{p_i,q_i}\right)\times \left(\prod_{j\in J} N^{C_2} SP_n S^{r_j} \right)
\]
is also a finite Galois-Maximal space.
Moreover, the sequence
\[
Z_0 \longrightarrow Z_1 \longrightarrow Z_2 \longrightarrow \cdots \longrightarrow Z_n \longrightarrow Z_{n+1}\longrightarrow \cdots
\]
with the obvious maps
\begin{align*}
\varphi_n = \left( \prod_{i\in I}\iota^i_{n,n+1} \right) \times \left( \prod_{j\in J}N^{C_2}\lambda^j_{n,n+1} \right), \qquad &\iota^i_{n,n+1}:SP_nS^{p_i,q_i}\hookrightarrow SP_nS^{p_i,q_i},\\
&\lambda_{n,n+1}^j: SP_nS^{r_j} \hookrightarrow SP_{n+1}S^{r_j},
\end{align*}
exhibits cohomological stability, as follows from Property \ref{prpty:products_cohomological_properties}.
Thus,
\[
Z= \bigcup_n Z_n \cong SP_\infty W
\]
is a Galois-Maximal space, by virtue of Theorem \ref{thm:cohomological_properties}(\ref{item:cohomological_stability}) (notice that the isomorphism is only possible after swapping products and colimits).

Now that $SP_\infty W$ is known to be a Galois-Maximal space, we can use Theorem \ref{thm:cohomological_properties}(\ref{item:cohomological_splitting}) to retrieve the result for each $SP_nW$.
We only need to make sure that the hypotheses of this statement are verified, but this is exactly Lemma \ref{lem:symmetric_products}.

Finally, we address the general case.
If $X$ is any finite Galois-Maximal space, then Theorem \ref{thm:main_characterisation} applies and 
\[
\F_2[X] \simeq_{C_2} \left(\bigoplus_{i\in I} \F_2[S^{p_i,q_i}]\right)\oplus\left(\bigoplus_{j\in J} \F_2[S^{r_j}\wedge(C_2)_+]\right) \equiv \F_2[W],
\]
where $W$ is of the form \eqref{eq:W}.
We also apply the same characterisation theorem to each $SP_nW$:
\[
\F_2[SP_nW] \simeq_{C_2} \left(\bigoplus_{k\in K_n} \F_2[S^{p_k,q_k}]\right)\oplus\left(\bigoplus_{l\in L_n} \F_2[S^{r_l}\wedge(C_2)_+]\right)
\]
for some finite sets $K_n$ and $L_n$.
As $\F_2[X] \simeq_{C_2} \F_2[W]$, Lemma \ref{lem:functor_0-cycles} implies that $\F_2[SP_nX]\simeq_{C_2} \F_2[SP_nW]$, \ie,
\[
\F_2[SP_nX]\simeq_{C_2} \left(\bigoplus_{k\in K_n} \F_2[S^{p_k,q_k}]\right)\oplus\left(\bigoplus_{l\in L_n} \F_2[S^{r_l}\wedge(C_2)_+]\right).
\]
Hence, by Theorem \ref{thm:main_characterisation} again, we conclude that $SP_nX$ is also a Galois-Maximal space.
Now, the case $n=\infty$ follows from cohomological stability once again.
\end{proof}

\section*{Acknowledgments}

The results of the paper are part of the author's thesis.
During his PhD, the author was supported by a scholarship granted by the Portuguese institution Fundação para a Ciência e a Tecnologia (FCT), to which the author expresses his gratitude.
The author would also like to thank his advisors Pedro F.~dos Santos and Carlos Florentino for all their guidance and support, and Erwan Brugallé, Florent Schaffhauser and Clover for the useful conversations and ideas that they shared.

\bibliographystyle{acm}
\bibliography{data_base}

\end{document}